\documentclass[11pt]{article}
\usepackage{amsmath}
\usepackage{amsfonts}
\usepackage{amssymb}
\usepackage{amsthm}
\usepackage{qtree}
\usepackage{xcolor}
\usepackage{dsfont}
\usepackage{showlabels}
\usepackage{refcheck}
 \allowdisplaybreaks[4]
\newcommand{\N}{{\mathbb N}}
\newcommand{\M}{{\mathcal M}}

\newcommand{\R}{{\mathbb R}}
\newcommand{\LL}{{\mathcal L}}

\newtheorem{theorem}{Theorem}
\newtheorem{corollary}{Corollary}

\DeclareMathOperator{\var}{var}
\DeclareMathOperator{\diam}{diam}

\newtheorem{lemma}{Lemma}
\newtheorem{remark}{Remark}
\author{GUAN-ZHONG MA$^1$,\quad YAO XIAO$^{2*}$\\
{\small\em $^1$, Department of Mathematical Sciences, Tsinghua University, Beijing 100084, China}\\
{\small\em $^2$, Department of Mathematical Scineces, Tsinghua University, Beijing 100084, China}\\
{\small\em mgz09@mails.tsinghua.edu.cn, yaox11@mails.tsinghua.edu.cn}
}
\title{Multifractal analysis of dimension spectrum and the set of irregular points  in non-uniformly hyperbolic systems}
\date{}

\begin{document}
\maketitle
\begin{abstract}
We study the  multifractal analysis of dimension spectrum for almost additive potential in a class of one dimensional
non-uniformly hyperbolic dynamic systems and prove that the  irregular set has full Hausdroff dimension.
\end{abstract}
\smallskip



\section{Introduction}
 Given a compact metric space $X$, and $T$ a continuous transformation from $X$ to itself, we call the pairs  $(X,T)$ a topological dynamical systems.
A sequence $\Phi=(\phi_n)_{n=1}^\infty$ is said to be almost additive if every $\phi_n$ is continuous
from $X$ to $\mathbb{R}$ and there is a positive constant $C(\Phi)>0$ such that
$$-C(\Phi)+\phi_n(x)+\phi_p(T^n x)\leq \phi_{n+p}(x)\leq \phi_n(x)+\phi_p(T^n x)+C(\Phi),\ \forall n,p\in \mathbb{N},\forall x\in X.$$
We denote by $C_{aa}(X,T)$ the collection of almost additive potentials on $X$. The almost additive potential arise naturally in the study of non-conformal repellers \cite{BD2009} and topological pressure of product of positive matrices \cite{FengLau} .

If $\Phi=(\Phi^1,\cdots,\Phi^d)$ and  $\Phi^j\in C_{aa}(X,T)$ for each $j$,  we call $\Phi$ a  vector-valued
almost additive potential and write $\Phi\in \mathcal{C}_{aa}(X,T,d)$. For $\Phi\in \mathcal{C}_{aa}(X,T,d)$, we have $\Phi=(\phi_n)_{n=1}^\infty$ with
$\phi_n=(\phi_n^1,\cdots,\phi_n^d)$.

Given any $\Phi\in \mathcal{C}_{aa}(X,T,d)$, by subadditivity
we have $\Phi_*(\mu):=\lim\limits_{n\rightarrow\infty}\int_X \frac{\phi_n}{n}d\mu$ exists for every $\mu\in \mathcal{M}(X,T)$. We define the  set
$\mathcal{L}_{\Phi}=\{\Phi_*(\mu):\mu\in \mathcal{M}(X,T)\}$, which is compact and convex. Given $\Phi\in \mathcal{C}_{aa}(X,T,d)$ and $\alpha\in \mathbb{R}^d$, one can define the level set as $X_{\alpha}:=\{x\in X:\lim\limits_{n\rightarrow\infty}\frac{\phi_n(x)}{n}=\alpha\}$. It is well known that if $(X,T)$ satisfies specification condition,
then $X_{\alpha}\neq \emptyset$ if and only if $\alpha\in \mathcal{L}_{\Phi}$.
Roughly speaking, the level sets $X_{\alpha}$ forms multifractal decomposition and the map $\alpha\rightarrow \dim_{H} X_{\alpha}$ forms a multifractal spectrum.  We also define the set $X_{irr}=\left\{x\in X:\ \lim\limits_{n\rightarrow\infty} \frac{\Phi_n(x)}{n} \text{does not exist}\right\}$.

The theory of multifractal analysis for uniformly hyperbolic conformal dynamic system
is well developed in the aspects of entropy spectrum and Birkhoff spectrum and local dimension of Gibbs measure \cite{FFW,PW1997A,GP1997, PW1997B}. In the case of sub-shift of finite type, the multifractal analysis for the level sets of almost additive potential or quotient almost additve potential has been well understood \cite{BD2009,BQ}. However there is still not a complete picture for the multifractal analysis of non-uniform hyperbolic dynamic systems. In the recent years, people become more and more interested in the multifractal analysis of non-uniform hyperbolic dynamic systems \cite{GR2009, JJOP2010}.   In this note we proved that irregular set in non-uniform hyperbolic dynamic system carries full of Hausdroff dimension unless it is an empty set. The corresponding part in uniform hyperbolic dynamic systems was proved in \cite{BS2000,FFW}.  

We start with an introduction about the basic settings.
Let $T:\bigcup_{i=1}^m I_i\rightarrow [0,1]$ be a piecewise $C^{1}$ map satisfies the following condition:
\begin{itemize}
\item $I_i\subset [0,1], i=1,\cdots,m$  such that $I_{i}$ and $I_{j}$ does not overlap  for $i\ne j.$
\item $T|_{ I_j}:I_j\rightarrow [0,1]$ is onto and $C^{1}$ map, for all $1\leq j\leq m.$ There is a unique $x_j\in I_j$ such that $T(x_j)=x_j.$
\item   $T'(x)>1$ for $x\not\in\{x_1,\cdots, x_m\}$.
\end{itemize}
We remark that since the map $T$ is $C^1$, we have $T^\prime(x_j)\ge 1$ for $j=1,\cdots,m.$ If for some $j$,
   $T'(x_j)=1$, we call   $x_j$ a {\it parabolic} fixed point.

Define the attractor of $T$ as
$$
  \Lambda=\left\{x\in \bigcup_{j=1}^m I_j | T^{n}(x)\in [0,1], \forall n\geq0\right\}.
$$
It is known  that $\Lambda$ is  invariant under $T$ and we get a  dynamic system
$T: \Lambda \rightarrow \Lambda.$

This special class of non-uniform hyperbolic maps includes the famous example of Manneville-Pomeu map and Farey map \cite{PW1999}.

The above system has a symbolic coding which can be defined as follows.
Let  $T_{i}$ be the inverse map of $T|_{I_{i}}: I_{i}\rightarrow [0,1]$ for $i=1,\cdots,m$.
  Let $\mathcal{A}=\{1\dots, m\}$ and $\Sigma=\mathcal{A}^{\mathbb{N}}$. There is a shift map $\sigma:\Sigma\to\Sigma$ defined by $\sigma((\omega_n)_{n\ge 1})=(\omega_n)_{n\ge 2}$. Define a projection $\Pi: \Sigma \to [0,1]$ as
\[
\Pi(\omega)=\lim_{n\rightarrow\infty}T_{\omega_1}\circ T_{\omega_2}\circ \dots \circ T_{\omega_n}([0,1]).
\]
Then $\Pi(\Sigma)=\Lambda$ and moreover

$$
 \Pi\circ\sigma(\omega)=T\circ\Pi(\omega).
$$

Obviously, we have that $\Pi$ is a bijection except for at most countable points.

 In this paper, we concern with $\Lambda_{\alpha}$ and $X_{\alpha}$ respectively for
$\Phi\in \mathcal{C}_{aa}(\Lambda,T,d)$ and $\Psi\in \mathcal{C}_{aa}(\Sigma,\sigma,d)$.
 Two kinds of level set are related in the following way.  Given $\Phi\in \mathcal{C}_{aa}(\Lambda,T,d)$. Define $\Psi=\Phi\circ \Pi$, then $\Psi\in C_{aa}(\Sigma,T,d)$ and $\Pi(X_\alpha)=\Lambda_\alpha$.

Define  $g(\omega):=-\log T'_{\omega_1}\Pi(\sigma \omega)$ and let
$$
\tilde{\Sigma}=\left\{\omega\in\Sigma:\liminf\limits_{n\rightarrow \infty }\frac{1}{n}\sum_{j=0}^{n-1}g(\sigma^j\omega)>0\right\}.
$$
Let $h(\mu,\sigma)$, $\lambda(\mu,\sigma)$  be  the metrical entropy and Lyapunov exponent of $\mu$.  We have the following  theorem:
\begin{theorem}\cite{MaYao2013}\label{main-1}
Given $\Phi\in \mathcal{C}_{aa}(\Sigma,T,d)$, then for $\alpha\in \LL_{\Phi}$,
$$\dim_{\text{H}}\Pi(X_{\alpha}\cap \tilde{\Sigma})= \underset{\mu\in \mathcal{M}(\Sigma,\sigma)}{\sup}\left\{\;\frac{h(\mu,\sigma)}{\lambda(\mu,\sigma)}\;\left|\;\lim_{n\rightarrow\infty}\int\frac{\phi_{n}}{n}  d \mu=\alpha, \lambda(\mu,\sigma)>0\right.\right\}.
$$
\end{theorem}
If we take $d=1$ and $\phi_n=nc$ where $c$ is a real constant. We have the following result.
\begin{corollary}
Let $T:\Lambda\rightarrow \Lambda$ is non-uniform hyperbolic, we have
$$\dim_{\text{H}}\Pi( \tilde{\Sigma})= \underset{\mu\in \mathcal{M}(\Sigma,\sigma)}{\sup}\left\{\;\frac{h(\mu,\sigma)}{\lambda(\mu,\sigma)}\;\left|\, \lambda(\mu,\sigma)>0\right.\right\}.
$$
\end{corollary}
 Roughly speaking $\Pi\tilde{\Sigma}$  can be seen as the hyperbolic part of non-uniform hyperbolic attractor. Of course this corollary implies the following theorem in the uniform hyperbolic setting, for which  $\tilde{\Sigma}=\Sigma$. One has
\begin{theorem}\cite{BQ}
Assume that $T:\Lambda\rightarrow \Lambda$ is uniformly hyperbolic, then
$$\dim_{\text{H}}\Pi(\Sigma)= \underset{\mu\in \mathcal{M}(\Sigma,\sigma)}{\sup}\left\{\;\frac{h(\mu,\sigma)}{\lambda(\mu,\sigma)}\;\left|\, \lambda(\mu,\sigma)>0\right.\right\}.$$
\end{theorem}

In the non-uniform hyperbolic dynamic systems, it is really subtle that whether the hyperbolic part of attractor $\Pi\tilde{\Sigma}$ carries the full hausdroff dimension of the attractor of $\Pi\Sigma$. This has been verified in \cite{Urbanski1996} for the case that each inverse branch of $T$ is $C^{1+\alpha}$  together with some geometric conditions.
In \cite{GR2009}, it was proved that $\dim_{H}\Pi\tilde{\Sigma}=\dim_{H}\Pi\Sigma$ under $C^{1+Lip}$ condition . However,  as pointed  in \cite{JJOP2010}, it is still unknown whether it is true for $C^{1}$ condition in non-uniform hyperbolic  dynamic systems. The following assumptions implied that $\dim_{H}\Pi\tilde{\Sigma}=\dim_{H}\Pi\Sigma$, which was  first proposed in \cite{JJOP2010}.

\noindent {\bf Assumptions:}
For any $\epsilon>0$, there exists $\nu\in\M(\Lambda,T)$ such that $\lambda(\nu,T)>0$ and
$\frac{h(\nu,T)}{\lambda(\nu,T)}>\dim_{\text{H}}\Lambda-\epsilon$.

Consider the system $T: \Lambda\to\Lambda.$
Let $I\subset \{x_1,\cdots, x_m\}$ be the set of parabolic fixed points. Given $\Phi\in \mathcal{C}_{aa}(\Lambda,T,d)$, we define $A=\text{Co}\left\{\lim\limits_{n\rightarrow\infty}\frac{\phi_n(x)}{n}:x\in\mathcal{I}\right\}$, which is  the {\it convex hull} of
$\left\{\lim\limits_{n\rightarrow\infty}\frac{\phi_n(x)}{n}:x\in\mathcal{I}\right\}$.

\begin{theorem}\cite{MaYao2013}\label{main-3}
 Let$(\Lambda,T)$ be a system defined as above. Given $\Phi\in \mathcal{C}_{aa}(\Lambda,T,d)$ and define $A$  as above. Under the assumption above,  then  for any $\alpha\in \LL_\Phi \setminus A$, we have
$$
\dim_{\text{H}}\Lambda_{\alpha}= \underset{\mu\in \mathcal{M}(\Lambda,T)}{\sup}\left\{\frac{h(\mu,T)}{\lambda(\mu,T)}\left|\Phi_*(\mu)=\alpha\right. \right\},
$$
and   for all $\alpha\in A$ we have $\dim_{\text{H}}\Lambda_{\alpha}=\dim_{\text{H}}\Lambda$.
\end{theorem}

Finally we consider the Hausdroff dimension of $\Lambda_{irr}$. By Kingman's sub-additive ergodic Theorem, we have $\mu(\Lambda_{irr})=0$ for any $\mu\in\mathcal{M}(\Lambda,T)$. However, this set carries full topological entropy and full Hausdorff dimension in most cases\cite{BS2000}, especially in uniformly hyperbolic dynamics.  Motivated by the method in \cite{JJOP2010,FLW2002},  we can  proved it is also true for non-uniform hyperbolic dynamic systems in a simple way.
\begin{theorem}\label{main-4}
Under the assumption in Theorem \ref{main-3} and assume that $\sharp\mathcal{L}_{\Phi}\geq 2$, then $\dim_H\Lambda_{irr}=\dim_H\Lambda$.
\end{theorem}

\begin{remark}
In \cite{MaYao2013}, Theorem\ref{main-1} and Theorem \ref{main-3} are proved for the case of additive potential in higher dimension. The skills there can be completely extended to the almost additive potential.
\end{remark}

\begin{remark}
There has been great interest in the study of irregular set in recent trend \cite{BD2009,Olsen, FFW, FLW2002,Thompson2010}. The full dimension of irregular set has been verified in subshifit of finite type\cite{BS2000, FFW}, conformal repellers\cite{BD2009, FFW}. It is interesting to ask the corresponding question in the non-uniform hyperbolic dynamic systems. It is possible to follow the line in \cite{FLW2002} and use the approximation skills as \cite{GR2009} to give a proof for the Hausdroff dimension of irregular set. Here we combine some ideas in \cite{BS2000, FLW2002, JJOP2010} to give a short and direct proof.
\end{remark}

The rest of this  note is organized as follows.  In Section \ref{preliminary} we give  some preliminary results  and lemmas which are needed for the proof.
In Section \ref{thm4}, we prove Theorem \ref{main-4}.

\section{Preliminaries}\label{preliminary}

In this section, we will give the notations and the lemmas needed in the proof.

Assume that $T:X\to X$ is a topological dynamical system.
Denote by $\M(X,T)$ the set of all invariant probability measures and ${\mathcal E}(X,T)$ the set of all ergodic probability measures. Given $\mu\in\M(X,T)$, let $h(\mu,T)$ be the metric entropy of $\mu$.

Recall that  $\mathcal{A}=\{1,2\dots m\}$ and $\Sigma=\mathcal{A}^{\mathbb{N}}$.  Write
$
\Sigma_n=\{w=w_1\cdots w_n: w_i\in\mathcal{A}\}.
$
 For $\omega=\{\omega_{n}\}_{n=1}^{\infty}\in \Sigma$, write $\omega|_n=\omega_1\cdots \omega_n$.
 For $w\in\Sigma_n$  define the cylinder $[w]:=\{\omega: \omega|_n=w\}$.

 If $\phi:\Sigma\rightarrow \mathbb{R}^d$ is continuous,we define the {\it $n$-th variation} of $\Phi$  as$$||\phi||_n:=\sup\limits_{\omega|_{n}=\tau|_{n}}|\phi(\omega)-\phi(\tau)|.$$
For $\Phi\in\mathcal{C}_{aa}(\Sigma,T,d)$, we define $||\Phi||_n:=||\phi_n||_n$.

where $|\cdot|$  is the Euclidean norm in $\R^{d}$. Given $f: \Sigma\rightarrow \R^d$ continuous, let
$\|f\|:=\underset{\tau\in\Sigma}{\sup}|f(\tau)|$.  For $f: \Lambda\rightarrow \R^d$ continuous we define $\|f\|$ similarly.
 We have the following standard result:

\begin{lemma}\label{variation}
If $\Phi=\{\phi_{n}\}_{n=1}^{\infty}\in \mathcal{C}_{aa}(\Sigma,T,d)$, then $\lim\limits_{n\rightarrow\infty}\frac{1}{n}\|\Phi\|_{n}=0$
\end{lemma}
Consider the projection $\Pi: \Sigma\to \Lambda$. Let
$\tilde \Lambda:=\{x\in \Lambda: \# \{\Pi^{-1}(x)\}=2\}$.
In other words $\tilde \Lambda$ is the set of such $x$ with two codings. By our assumption on $I_j$, we know that both  $\tilde \Lambda$ and $\Pi^{-1}\tilde \Lambda$ are  at most countable. Moreover
\begin{equation}\label{coding}
\Pi^{-1}\tilde \Lambda\subset \{\omega: \omega= wm^\infty \text{ or } \tilde w 1^\infty\}.
\end{equation}

Then it is seen that
$$
\Pi: \Sigma\setminus \Pi^{-1}(\tilde \Lambda)\to \Lambda\setminus \tilde \Lambda
$$
is a bijection. We will need this fact in the proof of the lower bound of Theorem \ref{main-1}.

For $w=w_1\cdots w_n$, write $I_w=T_{w_1}\circ \dots \circ T_{w_n}[0,1].$ Especially for $\omega\in\Sigma$, we write $I_n(\omega)=I_{\omega|_n}$.
Let $D_{n}(\omega)=\diam(I_{n}(\omega))$.
Recall that we have defined
  $g(\omega):=-\log T'_{\omega_1}\Pi(\sigma \omega)$ and

$D_n(\omega)$ can be estimated via $A_ng(\omega)$ by the following lemma:
\begin{lemma}[\cite{Urbanski1996,JJOP2010}]\label{appro}
Under the assumption on $T$, $D_{n}(\omega)$  converges to $0$ uniformly. Moreover
$$
\lim\limits_{n\rightarrow \infty}\sup\limits_{\omega\in \Sigma}\left\{|-\frac{1}{n}\log D_{n}(\omega)-A_{n}g(\omega)|\right\}=0.
$$
\end{lemma}
 By this lemma we can understand that $\tilde \Sigma$ is the set  of such points $\omega$ such that the length of $I_n(\omega)$ tends to 0 exponentially. To simplify the notation we write $\tilde \lambda_n(\omega)=-\log D_n(\omega)/n.$

 Given $\mu\in\M(\Sigma,\sigma)$,  let
 $\lambda(\mu,\sigma):=\int g d\mu$ be the Lyapunov exponent of $\mu.$ Similarly given $\mu\in\M(\Lambda,T)$, let
 $\lambda(\mu,T):=\int \log|T^\prime| d\mu$ be the Lyapunov exponent of $\mu.$ For a $\mu\in\M(\Sigma,\sigma)$, we denote the image of $\mu$ under $\Pi$ by $\Pi_*\mu$.

The following lemma, which is a combination of Lemma 2 and Lemma 3 in \cite{JJOP2010}, is very useful in our proof.

\begin{lemma}\label{basic}
For any $\mu\in \mathcal{M}(\Sigma,\sigma)$, there exists  a sequence of ergodic measures $\{\mu_n:n\ge 1\}$ such that
  $\mu_n\rightarrow \mu$   in the weak star topology and
 $$
 h(\mu_n,\sigma)\rightarrow h(\mu,\sigma);
 \ \ \ \lambda(\mu_n,\sigma)\rightarrow \lambda(\mu,\sigma).
 $$
\end{lemma}

We remark that from their proof each  ergodic measure $\mu_n$ is continuous, i.e. $\mu_n$ has no atom.

\begin{lemma}\label{continuity}
Assume that $\Phi\in \mathcal{C}_{aa}(\Sigma,\sigma,d)$, and given a sequence of measures $\{\mu_n\}_{n=1}^{\infty}$, such that $\lim\limits_{n\rightarrow\infty}\mu_n=\mu\in \mathcal{M}(\Sigma,\sigma)$, then
$\lim\limits_{n\rightarrow\infty}\lim\limits_{m\rightarrow\infty}\frac{1}{m}\int\phi_{m}d\mu_{n}=\lim\limits_{m\rightarrow\infty}\frac{1}{m}\int \phi_{m}d\mu$,
i.e, $\lim\limits_n\Phi_*(\mu_n)=\Phi_*(\mu)$.
\end{lemma}
\begin{proof}
Let $C$ be a constant vector with each coordinate positive such that
\[-C+\phi_{n}(T^p x)+\phi_{p}(x)\leq \phi_{n+p}(x)\leq C+\phi_{n}(T^p x)+\phi_{p}(x)\]
$\forall n,p\in \mathbb{N}, x\in X$.
By sub-additivity of the families $\{\phi_{m}(x)+C\}_{m=1}^{\infty}$ and $\{\phi_{m}(x)-C\}_{m=1}^{\infty}$, we have
\[\lim\limits_{m\rightarrow\infty}\frac{1}{m}\int(\phi_{m}-C)d\mu_{n}=\lim\limits_{m\rightarrow\infty}\frac{1}{m}\int
\phi_{m}d\mu_{n}=\lim\limits_{m\rightarrow\infty}\frac{1}{m}\int(\phi_{m}+C)d\mu_{n}\]
and
\[\sup_{m}\frac{1}{m}\int(\phi_{m}-C)d\mu_{n}=\lim\limits_{m\rightarrow\infty}\frac{1}{m}\int
\phi_{m}d\mu_{n}=\inf_{m}\frac{1}{m}\int(\phi_{m}+C)d\mu_{n}.\]
Thus we get
\[\frac{1}{m}\int(\phi_{m}-C)d\mu_{n}=\lim\limits_{m\rightarrow\infty}\frac{1}{m}\int
\phi_{m}d\mu_{n}=\frac{1}{m}\int(\phi_{m}+C)d\mu_{n}.\]
Then taking $n$ goes to infinity, and $m$ goes to infinity, we get the desired result.
\end{proof}

\section {\bf Proof of Irregular set Theorem \ref{main-4}.}\label{thm4}
For $\Phi\in \mathcal{C}_{aa}(\Lambda,T,d)$, we define $\Psi=\Phi\circ\Pi$.  It is rather easy to check $\mathcal{L}_\Phi=\mathcal{L}_\Psi$.
Then Theorem \ref{main-4} is a immediately consequence of the following Lemma.
\begin{lemma}\label{deviation}
For any $\mu, \nu\in \mathcal{M}(\sigma,\Sigma)$ with $\lambda(\mu,\sigma)>0$,  $\lambda(\nu,\sigma)>0$ and $\Psi_{*}(\mu)\neq \Psi_{*}(\nu)$,  we have
\[\dim_{H}\Lambda_{irr}\geq \min\left\{\frac{h(\mu,\sigma)}{\lambda(\mu,\sigma)}, \frac{h(\nu,\sigma)}{\lambda(\nu,\sigma)}\right\}.\]
\end{lemma}
\noindent {\bf Proof of Theorem \ref{main-4}.}
Under the assumption of Theorem \ref{main-3}, $$\dim_{H}\Lambda={\sup\limits_{\mu\in\mathcal{M}(\Sigma,\sigma)}}\left\{\frac{h(\mu,\sigma)}{\lambda(\mu,\sigma)}:\lambda(\mu,\sigma)>0\right\}.$$

For any $\epsilon>0$, there exists $\mu\in \mathcal{M}(\sigma,\Sigma)$ such that
$\frac{h(\mu,\sigma)}{\lambda(\mu,\sigma)}\geq \dim_{H}\Lambda-\epsilon$. Write $\alpha=\Psi_{*}(\mu)$. Since $\sharp \mathcal{L}_{\Psi}\geq 2$, we can choose $\nu\in \mathcal{M}(\sigma,\Sigma)$ such that $\Psi_{*}(\nu)=\beta\neq \alpha$.

Define $\nu_{s}=s\mu+(1-s)\nu$, where $s\in [0,1]$. We have $\Psi_{*}(\mu_s)=s\alpha+(1-s)\beta\neq \alpha$ for any $s\in [0,1)$. By Lemma \ref{deviation},
\[\dim_{H}\Lambda_{irr}\geq \min\left\{\frac{h(\mu,\sigma)}{\lambda(\mu,\sigma)}, \frac{h(\mu_s,\sigma)}{\lambda(\mu_s,\sigma)}\right\}=\min\left\{\frac{h(\mu,\sigma)}{\lambda(\mu,\sigma)}, \frac{sh(\mu,\sigma)+(1-s)h(\nu,\sigma)}{s\lambda(\mu,\sigma)+(1-s)\lambda(\nu,\sigma)}\right\}\]
for all $s\in [0,1)$.
Taking $s$ goes to $1$, we get $\dim_{H}\Lambda_{irr}\geq \frac{h(\mu,\sigma)}{\lambda(\mu,\sigma)} \geq \dim_{H}\Lambda-\epsilon$.
By the arbitrary of $\epsilon$, we get the desired result.

\noindent {\bf Proof of Lemma \ref{deviation}.}
 By Lemma \ref{variation} and Lemma \ref{appro}, we can choose a decreasing sequence $\epsilon_{i}\downarrow 0$   such that for all $n\geq 2i-1$,
\begin{equation}\label{var2}
\frac{1}{n} ||\Psi||_{n}<\epsilon_{2i-1},\ \ \ \var_n A_n g<\epsilon_{2i-1}\ \ \text{ and }\ \  \ |\tilde{\lambda}_n(\omega)-A_n g(\omega)|<\epsilon_{2i-1} (\forall \omega\in\Sigma).
\end{equation}
By Lemma \ref{basic} and Lemma \ref{continuity}, we can choose a sequence of  $\mu_{2i-1}\in\mathcal{E}(\Sigma,\sigma)$, such that
\begin{equation}\label{control1}
|\Psi_{*}(\mu_{2i-1})-\alpha|<\epsilon_{2i-1},\
|h(\mu_{2i-1},\sigma)-h(\mu,\sigma)|<\epsilon_{2i-1}\ \text{ and }\
|\lambda(\mu_{2i-1},\sigma)-\lambda(\mu,\sigma)|<\epsilon_{2i-1}.
\end{equation}
Since $\mu_{2i-1}$ is ergodic,  for $\mu_{2i-1}$ a.e. $\omega$,
\begin{equation}\label{block}
\frac{1}{n} \Psi_{n}(\omega)\to \Psi_{*}(\mu_{2i-1}),  \
A_n g(\omega)\to \lambda(\mu_{2i-1},\sigma)  \text{ and }
-\frac{\log \mu_{2i-1}[\omega|_n]}{n}\to h(\mu_{2i-1},\sigma).
\end{equation}

Fix $\delta>0$. Since $\mu_{2i-1}$ is continuous as we remarked after Lemma \ref{basic}, there exists $\ell_{2i-1}\geq 2i-1$ such that
$\mu_{2i-1}(\bigcup_{j=1}^m [j^{\ell_{2i-1}}])\le \delta/2.$
By  Egorov's theorem,  there exists $\Omega'(2i-1)\subset \Sigma$ such that $\mu_{2i-1}(\Omega'(2i-1))>1-\delta/2$ and \eqref{block}
 holds uniformly on $\Omega'(2i-1)$. Then there  exists   $l_{2i-1}\geq \ell_{2i-1}\geq 2i-1$ such that  for all  $n\geq l_{2i-1}$  and $\omega\in\Omega'(2i-1)$, we
have
\begin{equation}\label{estimation2}
\begin{cases}
|\frac{1}{n} \Psi_{n}(\omega)- \Psi_{*}(\mu_{2i-1})|<\epsilon_{2i-1} \\
|A_n g(\omega)- \lambda(\mu_{2i-1},\sigma)|< \epsilon_{2i-1} \\
|-{\log\mu_{2i-1}[\omega|_n]}/{n} - h(\mu_{2i-1},\sigma)|< \epsilon_{2i-1}
\end{cases}
\end{equation}
Let
$$
\Sigma(2i-1)=\{\omega|_{l_{2i-1}}\ |\ \omega\in \Omega'(2i-1)\}\setminus\{1^{l_{2i-1}},\cdots, m^{l_{2i-1}}\}.
$$
Let ${\Omega}(2i-1)=\bigcup_{w\in \Sigma(2i-1)}[w]$. Then
$$
\mu_{2i-1}({\Omega}(2i-1))\ge \mu_{2i-1}(\Omega^\prime(2i-1))-\mu_{2i-1}(\bigcup_{j=1}^m [j^{l_{2i-1}}]) \ge 1-\delta/2 -\delta/2=1-\delta.
$$
Similarly for all $n\geq 2i$, we have
\begin{equation}\label{var3}
\frac{1}{n}||\Psi||_{n}<\epsilon_{2i},\ \ \ \var_n A_n g<\epsilon_{2i}\ \ \text{ and }\ \  \ |\tilde{\lambda}_n(\omega)-A_n g(\omega)|<\epsilon_{2i} (\forall \omega\in\Sigma).
\end{equation}
By Lemma \ref{basic} we can pick a sequence of  $\nu_{2i}\in\mathcal{E}(\Sigma,\sigma)$, such that
\begin{equation}\label{control2}
|\Psi_{*}(\nu_{2i})-\alpha|<\epsilon_{2i},\
|h(\nu_{2i},\sigma)-h(\nu,\sigma)|<\epsilon_{2i}\ \text{ and }\
|\lambda(\nu_{2i},\sigma)-\lambda(\nu,\sigma)|<\epsilon_{2i}.
\end{equation}
Since $\nu_{2i}$ is ergodic,  for $\nu_{2i}$ a.e. $\omega$,
\begin{equation}\label{block}
\frac{1}{n} \Psi_{n}(\omega)\to \Psi_{*}(\nu_{2i}),  \
A_n g(\omega)\to \lambda(\nu_{2i},\sigma)  \text{ and }
-\frac{\log \nu_{2i}[\omega|_n]}{n}\to h(\mu_{2i},\sigma).
\end{equation}

Similarly for all $n\ge 2i$, we have
\begin{equation}\label{estimation3}
\begin{cases}
|\frac{1}{n} \Psi_{n}(\omega)- \Psi_{*}(\nu_{2i})|<\epsilon_{2i} \\
|A_n g(\omega)- \lambda(\nu_{2i},\sigma)|< \epsilon_{2i-1} \\
|-{\log\nu_{2i}[\omega|_n]}/{n} - h(\nu_{2i},\sigma)|< \epsilon_{2i}
\end{cases}
\end{equation}
Let
$$
\Sigma(2i)=\{\omega|_{l_{2i}}\ |\ \omega\in \Omega'(2i)\}\setminus\{1^{l_{2i}},\cdots, m^{l_{2i}}\}.
$$
Let ${\Omega}(2i)=\bigcup_{w\in \Sigma(2i)}[w]$. Then
$$
\nu_{2i}({\Omega}(2i))\ge \nu_{2i}(\Omega^\prime(i))-\nu_{2i}(\bigcup_{j=1}^m [j^{l_{2i}}]) \ge 1-\delta/2 -\delta/2=1-\delta.
$$

It is seen that we can take $l_i$ such that  $l_i\uparrow \infty$ and still satisfies all the above property.
Let $N_{0}=1$, $N_i=2^{l_{i+2}+N_{i-1}}$, $i\geq 1$. Let
$$
M=\overset{\infty}{\underset{i=1}{\prod}}\overset{N_i}{\underset{j=1}{\prod}}\Sigma(i).
$$
By the definition of $\Sigma(i)$ and \eqref{coding}, it is ready to see that
$M\cap \Pi^{-1}\tilde \Lambda =\emptyset.$ In the following we will  show that $\Pi M\subset \Lambda_{irr}.$ To be precise, we will check the following result:
\begin{lemma}\label{limits}
Let $n_{j}=\sum\limits_{i=1}^{j}l_{i}N_{i}$ and fix $\omega\in M$, then we have
\item $\lim\limits_{j\rightarrow\infty}\frac{\Psi_{n_{2j+1}}(\omega)}{n_{2j+1}}=\alpha$,\\
\item $\lim\limits_{j\rightarrow\infty}\frac{\Psi_{n_{2j}}(\omega)}{n_{2j}}=\beta$.
\end{lemma}
\noindent{\bf Proof of Lemma\ref{limits}}

\begin{align*}
&\Psi_{n_{2j+1}} (\omega)-n_{2j+1}\alpha\\
\leq &\underset{i=1}{\overset{2j+1}{\sum}}\underset{k=0}{\overset{N_i-1}{\sum}}[\Psi_{l_i}
(\sigma^{n_{i-1}+kl_{i}}\omega)-l_i\alpha+C] \\
=&\underset{i=1}{\overset{j+1}{\sum}}\underset{k=0}{\overset{N_{2i-1}-1}{\sum}}[\Psi_{l_{2i-1}}
(\sigma^{n_{2i-2}kl_{2i-1}}\omega)-l_{2i-1}\alpha+C]+\underset{i=1}{\overset{j}{\sum}}\underset{k=0}{\overset{N_{2i}-1}{\sum}}[\Psi_{l_{2i}}
(\sigma^{n_{2i-1}kl_{2i}}\omega)-l_{2i}\alpha+C] \\
\leq & \underset{i=1}{\overset{j+1}{\sum}}3l_{2i-1}N_{2i-1}\epsilon_{2i-1}\vec{1}
+\underset{i=1}{\overset{j}{\sum}}[3l_{2i}N_{2i}\epsilon_{2i}\vec{1}+l_{2i}N_{2i}(\beta-\alpha)]+\sum\limits_{i=1}^{2j+1}N_iC\\
=&\underset{i=1}{\overset{2j+1}{\sum}}N_i(3l_{i}\epsilon_{i}\vec{1}+C)+\underset{i=1}{\overset{j}{\sum}}l_{2i}N_{2i}(\beta-\alpha).
\end{align*}
where for the second inequality we use (\ref{var2}) (\ref{control1}) (\ref{estimation2}) (\ref{var3}) (\ref{control2})
(\ref{estimation3})and similar method used in the proof of lower bound of Theorem 1.
Similarly we have
$$\Psi_{n_{2j+1}} (\omega)-n_{2j+1}\alpha\ge -\underset{i=1}{\overset{2j+1}{\sum}}N_i(3l_{i}\epsilon_{i}\vec{1}+C)+\underset{i=1}{\overset{j}{\sum}}l_{2i}N_{2i}(\beta-\alpha).$$
Noting that
\[\lim_{j\rightarrow\infty}\frac{l_{2}N_{2}+l_{4}N_{4}+\cdots +l_{2j}N_{2j}}{l_{1}N_{1}+l_{2}N_{2}+\cdots +l_{2j+1}N_{2j+1}}=0,\]
we have
$$\lim\limits_{j\rightarrow\infty}\frac{\Psi_{n_{2j-1}}(\omega)}{n_{2j-1}}=\alpha.$$ Similarly we can also get
$$\lim\limits_{j\rightarrow\infty}\frac{\Psi_{n_{2j}}(\omega)}{n_{2j}}=\beta.$$
This implies that $\Pi M\subset \Lambda_{irr}$.

Now we will construct a measure $\eta$ supported on $M$ and show that for all $x\in \Pi(M)$
$$
\underset{r\downarrow0}{\liminf}\frac{\log \Pi_*\eta(B(x,r))}{\log r}\geq
 \min\left\{\frac{h(\mu,\sigma)}{\lambda(\mu,\sigma)}, \frac{h(\nu,\sigma)}{\lambda(\nu,\sigma)}\right\}.
 $$
Consequently, we have
\[\dim_{H}\Lambda_{irr}\geq \dim_{H}\Pi M \ge\min\left\{\frac{h(\mu,\sigma)}{\lambda(\mu,\sigma)}, \frac{h(\nu,\sigma)}{\lambda(\nu,\sigma)}\right\}.\]
Then the result follows.\hfill

For convenience we relabel the following sequence
$$
\underbrace{l_1\cdots l_1,}_{N_1}\cdots,\underbrace{l_i\cdots l_i,}_{N_i}\cdots
$$
 as $\{l^*_i:i\ge 1\}$.
Relabel the following sequence
$$
\underbrace{\Sigma(1)\cdots \Sigma(1),}_{N_1}\cdots,\underbrace{\Sigma(i)\cdots
\Sigma(i),}_{N_i}\cdots
$$
 as $\{\Sigma^{*}(i): i\ge 1\}$.
Accordingly  we get $\{{\Omega}'^*(i)\}$, $\{{\Omega}^*(i)\}$, $\{\nu^*_i)\}$, $\{\epsilon^*_i\}$. Let
$n_k=\underset{i=1}{\overset{k}{\sum}}l_i^*$. For any $n>0$, there  exists $J(n)\in\N$ such that
$\underset{i=1}{\overset{J(n)}{\sum}}l_i^*\leq n<\underset{i=1}{\overset{J(n)+1}{\sum}}l_i^*$.
There also exists $r(n)\in\N$ such that $\underset{i=1}{\overset{r(n)}{\sum}}N_i\leq
J(n)<\underset{i=1}{\overset{r(n)+1}{\sum}}N_i$. It is seen that
\begin{equation}\label{J-n}
J(n)\leq J(n+1)\leq J(n)+1,\ l_{J(n)+1}^*=l_{r(n)+1}\ \text{ and }\ l_{J(n)+2}^*\leq l_{r(n)+2},
\end{equation}
then, for $j=1,2$, $$\frac{l_{J(n)+j}^*}{\underset{i=1}{\overset{J(n)}{\sum}}l_i^*}\leq
\frac{l_{r(n)+j}}{N_{r(n)}l_{r(n)}}=\frac{l_{r(n)+j}}{2^{N_{r(n)-1}+l_{r(n)+2}}l_{r(n)}}.$$
We have
\begin{equation}\label{l-i-basic1}
{\underset{i=1}{\overset{J(n)+1}{\sum}}l_i^*}/{\underset{i=1}{\overset{J(n)}{\sum}}l_i^*}\to 1\ \ \text{ and }\ \ \ {l^*_{J(n)+j}}/{\underset{i=1}{\overset{J(n)}{\sum}}l_i^*}\to 0, j=1,2.
\end{equation}

For convenience, define $\eta_{i}=\mu_{i}$ if $i$ is odd, and $\eta_{i}=\nu_{i}$, if $i$ is even.
At first  we define a probability $m$ supported on $M$. For each  $w\in \Sigma^*(i)$ define
$$
\rho^i_w=\frac{\eta_{i}^\ast[w]}{\eta_{i}^\ast(\Omega^*(i))}.
$$
It is seen that $\sum_{w\in \Sigma^*(i)}\rho^i_w=1.$
Write
$
{\mathcal C}_n:=\{[w]:w\in \prod_{i=1}^n \Sigma^\ast(i)\}.
$
It is seen that $\sigma({\mathcal C}_n: n\ge 1)$ gives the Borel-$\sigma$ algebra in $M.$
For each $w=w_1\cdots w_n\in {\mathcal C}_n$ define
$$
\hat{\eta}([w])=\prod_{i=1}^n \rho^i_{w_i}.
$$
Let $\eta$ be the Kolmogorov extension of $\hat{\eta}$ to all the Borel sets.   By the construction it is seen that $\eta$ is supported on $M.$

Fix  $\omega\in M$. At first we find a lower bound for $D_n(\omega)$. Define $n_0=0$, $n_i=\sum\limits_{j=1}^il_j^*$, for $i\geq1$.
Recall  that $D_n(\omega)=e^{-n\tilde \lambda_n(\omega)}$. By the construction of $M$ we have $\sigma^{n_{i-1}}\omega \in [w]$ for some $w\in\Sigma^*(i)$, consequently there exists $\omega^i\in \Omega^{\prime \ast}(i)\cap [w]$ such that \eqref{estimation2} \eqref{estimation3} holds. By the similar method used in Theorem \ref{main-1}, we have
\begin{align*}
&  n\tilde \lambda_n(\omega)\\
\leq& n(A_ng(\omega)+\epsilon_{J(n)}^\ast)\\
\leq&
 \underset{i=1}{\overset{J(n)}{\sum}}l_i^*(A_{l_i^*}g (\sigma^{n_{i-1}}\omega)+\epsilon_{i}^*)
 +(n-n_{J(n)})\left(A_{n-n_{J(n)}}g (\sigma^{n_{J(n)}}\omega)+\epsilon_{J(n)}^*\right)\\
 \leq&
\underset{i=1}{\overset{J(n)}{\sum}}l_i^*\big\{A_{l_i^*}g(\sigma^{n_{i-1}}\omega)-A_{l_i^*}g(\omega^i)+A_{l_i^*}g
(\omega^i)-\lambda(\eta_{i},\sigma)+\\
&\lambda(\eta_{i},\sigma)-\lambda(\eta_i,\sigma)
+\lambda(\eta_i,\sigma)+\epsilon_{i}^*\big\}
+l_{J(n)+1}^*(||g||+\epsilon_{J(n)}^*)\\
 \leq &\underset{i=1}{\overset{J(n)}{\sum}}l_i^*(\lambda(\eta_i,\sigma)+
4\epsilon_{i}^*)+
l_{J(n)+1}^*(||g||+\epsilon_{J(n)}^*)=:\rho(n).
\end{align*}
Then
$D_n(\omega)\geq e^{-\rho(n)}.$ It is seen that $\rho(n)$ is increasing.

Now fix $x\in \Pi(M)$ and some $r>0$ small. Then there exists a unique $n=n_r$ such that
\begin{equation}\label{r}
e^{-\rho(n+1)}\le r< e^{-\rho(n)}.
\end{equation}
Consider the set of $n$-cylinders
$$
{\mathcal C}:=\{I_n(\omega): \omega\in M \text{ and } I_n(\omega)\cap B(x,r)\ne \emptyset\}.
$$
By the bound $D_n(\omega)\ge e^{-\rho(n)}$, the above  set consists of at most three cylinders, i.e. $\#{\mathcal C}\le 3$.

Choose $\omega\in M$ such that  $I_n(\omega)\in {\mathcal C}$. Write $\omega|_n=w_1\cdots w_{J(n)} v$, then $w_i\in \Sigma^*(i)$ and $v$ is a prefix of some $\tilde v\in \Sigma^*(J(n)+1)$.
Then
\begin{align*}
 \Pi_*\eta(I_n(\omega))=\nu[\omega|_n]
=&\underset{i=1}{\overset{J(n)}{\prod}}\frac{\eta_{i}^*[w_i]}{\eta_{i}^*(\Omega^*(i))}\cdot\frac{\eta_{J(n)+1}^*[v]}{\eta_{J(n)+1}^*(\Omega^*(J(n)+1))}\\
& \leq
(1-\delta)^{-J(n)-1}\underset{i=1}{\overset{J(n)}{\prod}}\eta_{i}^*[w_i].
\end{align*}
Then we conclude that $\Pi_*\eta(B(x,r))\le 3(1-\delta)^{-J(n)-1}\underset{i=1}{\overset{J(n)}{\prod}}\eta_i^*[w_i].$ For convenience, we define $\tau_i$ be the measure which is $\mu$ whenever $i$ is odd is $\nu$ whenever $i$ even. Consequently
\begin{align*}
&\quad \log \Pi_*\eta(B(x,r))\\
&\leq
-\underset{i=1}{\overset{J(n)}{\sum}}l_i^*\left(-\frac{\log\eta_{i}^*[w_i]}{l_i^*} \right)
-(J(n)+1)\log(1-\delta)+\log 3\\
&\leq
-\underset{i=1}{\overset{J(n)}{\sum}}l_i^*(h(\tau_i,\sigma)-2\epsilon_{i}^*)-(J(n)+1)\log(1-\delta)+\log 3,
\end{align*}
where for the second inequality  we use \eqref{control2} and \eqref{estimation3}.
Notice  that $r\rightarrow 0$ if and only if $n\rightarrow\infty$. By \eqref{J-n} we have $J(n+1)\le J(n)+1.$ Together with \eqref{r} and \eqref{l-i-basic1} we get
\begin{align*}
& \quad \underset{r\downarrow 0}{\liminf}\frac{\log\Pi_* \eta(B(x,r))}{\log r} \\
& \geq
\underset{n\rightarrow\infty}{\liminf}
\frac{\underset{i=1}{\overset{J(n)}{\sum}}l_i^*(h(\tau_i,\sigma)-2\epsilon_{i}^*)+(J(n)+1)\log(1-\delta)-\log 3}
{\underset{i=1}{\overset{J(n+1)}{\sum}}l_i^*(\lambda(\tau_i,\sigma)+4\epsilon_{i}^*)
+l_{J(n+1)+1}^*(||g||+\epsilon_{J(n+1)+1}^*)}\\
& =\liminf\limits_{n\rightarrow\infty}
\frac{\sum\limits_{i=1}^{J(n)}l_i^*(h(\tau_i,\sigma)-2\epsilon_i^*)}{\sum\limits_{i=1}^{J(n)}l_i^*(\lambda(\tau_i,\sigma)+4\epsilon_i^*)}\\
& \ge \lim\limits_{n\rightarrow\infty}
\frac{\sum\limits_{i=1}^{J(n)}l_i^*(\lambda(\tau_i,\sigma)+4\epsilon_i^*)
\min\{\frac{h(\mu,\sigma)-2\epsilon_i^*}{\lambda(\mu,\sigma)+4\epsilon_i^*},\frac{h(\nu,\sigma)-2\epsilon_i^*}{\lambda(\nu,\sigma)+4\epsilon_i^*}\}}
{\sum\limits_{i=1}^{J(n)}l_i^*(\lambda(\tau_i,\sigma)+4\epsilon_i^*)}\\
& =\lim\limits_n \min\left\{\frac{h(\mu,\sigma)-2\epsilon_J(n)^*}{\lambda(\mu,\sigma)+4\epsilon_J(n)^*},\frac{h(\nu,\sigma)-2\epsilon_J(n)^*}
{\lambda(\nu,\sigma)+4\epsilon_J(n)^*}\right\}\\
& =\min\left\{\frac{h(\mu,\sigma)}{\lambda(\mu,\sigma)},\frac{h(\nu,\sigma)}{\lambda(\nu,\sigma)}\right\}
\end{align*}

 $\Box$

\section{Acknowledgement}
We are very grateful that Yanhui Qu's great help during the preparations of this manuscript and especially for his patient discussions on the techniques of  constructions of Moran set.


\begin{thebibliography}{10}

\bibitem{BD2009} Barreira, L., \& Doutor, P. (2009). Almost additive multifractal analysis. Journal de math¨¦matiques pures et appliqu¨¦es, 92(1), 1-17.
\bibitem{BQ} J.Barral and Yan-Hui Qu. (2012) Localized asymptotic behavior for almost additive potentials. {\it Discrete Contin. Dyn. Syst.} 32 , no. 3, 717-751.
\bibitem{BS2000} Barreira, L., \& Schmeling, J. (2000). Sets of ¡°non-typical¡± points have full topological entropy and full Hausdorff dimension. Israel Journal of Mathematics, 116(1), 29-70.

\bibitem{FFW} Ai-Hua.Fan, De-Jun.Feng and  Jun.Wu. (2001) Recurrence, dimension and entropy. {\it  J. London Math. Soc. } (2) 64 , no. 1, 229-244.

\bibitem{FLW2002} De-Jun.Feng, Ka-sing.Lau \& Jun.Wu. (2002) Ergodic limits on the conformal Repellers.  {\it Advances in Mathematics}  169, 58-91.
\bibitem{GP1997}  Gatzouras, D., \& Peres, Y. (1997). Invariant measures of full dimension for some expanding maps. Ergodic theory and dynamical systems, 17(1), 147-167.
\bibitem{GR2009} Gelfert.K \& Rams.M. (2009) The Lyapunov spectrum of some parabolic systems. {\it Ergodic Theory and Dynamical Systems}, 29(3): 919-940.
\bibitem{GR2009} Gelfert.K \& Rams.M. (2009). Geometry of limit sets for expansive Markov systems.
Trans. Amer. Math. Soc. 361 , 2001-2020.
\bibitem{MaYao2013} Guan-Zhong. Ma, Xiao, Yao. (2013) Higher dimensional multifractal analysis of non-uniformly hyperbolic systems.
http://arxiv-web3.library.cornell.edu/abs/1311.5083
\bibitem{JJOP2010} A.Johansson, \& T. M. Jordan,   A. \"{O}berg and M. Pollicott. (2010). Multifractal analysis of non-uniformly hyperbolic systems. \textit{Israel Journal of Mathematics},177(1), 125-144.

\bibitem{Olsen} Olsen, L. (2003). Multifractal analysis of divergence points of deformed measure theoretical Birkhoff averages. Journal de math¨¦matiques pures et appliqu¨¦es, 82(12), 1591-1649.
\bibitem{PW1997A} Pesin, Y., \& Weiss, H. (1997). A multifractal analysis of equilibrium measures for conformal expanding maps and Moran-like geometric constructions. Journal of Statistical Physics, 86(1-2), 233-275.
\bibitem{PW1997B} Pesin, Y., \& Weiss, H. (1997). The multifractal analysis of Gibbs measures: motivation, mathematical foundation, and examples. Chaos: An Interdisciplinary Journal of Nonlinear Science, 7(1), 89-106.
\bibitem{PW1999}  Pollicott, M., \& Weiss, H. (1999). Multifractal analysis of Lyapunov exponent for continued fraction and Manneville¨CPomeau transformations and applications to Diophantine approximation. Communications in mathematical physics, 207(1), 145-171.

\bibitem{Thompson2010} Thompson, D. (2010). The irregular set for maps with the specification property has full topological pressure. Dynamical Systems, 25(1), 25-51.
\bibitem{Urbanski1996}  M.Urbanski. (1996). Parabolic Cantor sets.  {\it Fundamenta Mathematicae,} 151(3),  241-277.



\end{thebibliography}
\end{document}